\documentclass[12pt,reqno]{amsart}

\usepackage{amsmath,amsfonts,amsthm,amssymb,amsxtra,dsfont}
\usepackage{bbm} % to get \1
\usepackage[hidelinks]{hyperref} % creates hyperlinks to the references 

\newtheorem{theorem}{Theorem}%[section]

\newtheorem{lemma}[theorem]{Lemma}

\theoremstyle{definition}

\theoremstyle{remark}

\newtheorem{remark}[theorem]{Remark}

\newcommand{\1}{\mathds{1}}

\renewcommand{\epsilon}{\varepsilon}

\newcommand{\N}{\mathbb{N}}

\renewcommand{\phi}{\varphi}
\newcommand{\R}{\mathbb{R}}

\newcommand{\Sph}{\mathbb{S}}

\DeclareMathOperator{\dist}{dist}

\DeclareMathOperator{\dom}{dom}

\DeclareMathOperator{\ran}{ran}

\DeclareMathOperator{\Tr}{Tr}

\newcommand{\limplus}{{\mathchoice{\vcenter{\hbox{$\scriptstyle +$}}}
  {\vcenter{\hbox{$\scriptstyle +$}}}
  {\vcenter{\hbox{$\scriptscriptstyle +$}}}
  {\vcenter{\hbox{$\scriptscriptstyle +$}}}
}}
\newcommand{\limminus}{{\mathchoice{\vcenter{\hbox{$\scriptstyle -$}}}
  {\vcenter{\hbox{$\scriptstyle -$}}}
  {\vcenter{\hbox{$\scriptscriptstyle -$}}}
  {\vcenter{\hbox{$\scriptscriptstyle -$}}}
}}
\newcommand{\limpm}{{\mathchoice{\vcenter{\hbox{$\scriptstyle \pm$}}}
  {\vcenter{\hbox{$\scriptstyle \pm$}}}
  {\vcenter{\hbox{$\scriptscriptstyle \pm$}}}
  {\vcenter{\hbox{$\scriptscriptstyle \pm$}}}
}}

\begin{document}

\title[Improved semiclassical eigenvalue estimates]{Improved semiclassical eigenvalue estimates\\ for the Laplacian and the Landau Hamiltonian}

\author[R. L. Frank]{Rupert L. Frank}
\address[Rupert L. Frank]{Mathe\-matisches Institut, Ludwig-Maximilians Universit\"at M\"unchen, The\-resienstr.~39, 80333 M\"unchen, Germany, and Munich Center for Quantum Science and Technology, Schel\-ling\-str.~4, 80799 M\"unchen, Germany, and Mathematics 253-37, Caltech, Pasa\-de\-na, CA 91125, USA}
\email{r.frank@lmu.de}

\author[S. Larson]{Simon Larson}
\address{\textnormal{(Simon Larson)} Mathematical Sciences, Chalmers University of Technology and the University of Gothenburg, SE-41296 Gothenburg, Sweden}
\email{larsons@chalmers.se}

\author[P. Pfeiffer]{Paul Pfeiffer}
\address[Paul Pfeiffer]{Mathe\-matisches Institut, Ludwig-Maximilians Universit\"at M\"unchen, The\-resienstr.~39, 80333 M\"unchen, Germany}
\email{pfeiffer@math.lmu.de}

\thanks{
\copyright\, 2025 by the authors. This paper may be reproduced, in its entirety, for non-commercial purposes.\\
	Partial support through US National Science Foundation grant DMS-1954995 (R.L.F.), the German Research Foundation through EXC-2111-390814868 (R.L.F.) and TRR 352-Project-ID 470903074 (R.L.F. \& P.P.), as well as the Swedish Research Council grant no.~2023-03985 (S.L.) is acknowledged.}

\begin{abstract}
    The Berezin--Li--Yau and the Kr\"oger inequalities show that Riesz means of order $\geq 1$ of the eigenvalues of the Laplacian on a domain $\Omega$ of finite measure are bounded in terms of their semiclassical limit expressions. We show that these inequalities can be improved by a multiplicative factor that depends only on the dimension and the product $\sqrt\Lambda |\Omega|^{1/d}$, where $\Lambda$ is the eigenvalue cut-off parameter in the definition of the Riesz mean. The same holds when $|\Omega|^{1/d}$ is replaced by a generalized inradius of $\Omega$. Finally, we show similar inequalities in two dimensions in the presence of a constant magnetic field.
\end{abstract}

\maketitle

%%%%%%%%%%%%%%%%%%%%%%%%%%%%%%%%%%%%

\section{Introduction and main results}

Let $\Omega\subset\R^d$ be an open set of finite measure. We denote by $-\Delta_\Omega^{\rm D}$ and $-\Delta_\Omega^{\rm N}$ the Dirichlet and Neumann realizations of $-\Delta$ in $L^2(\Omega)$, defined via quadratic forms; see \cite[Section 3.1]{FrankLaptevWeidl}. The famous Weyl asymptotics state that for every $\gamma\geq 0$ and $\sharp\in\{\rm D, N\}$, one has
\begin{equation}
	\label{eq:weyl}
	\Tr (-\Delta_\Omega^\sharp -\Lambda)_\limminus^\gamma \sim L_{\gamma,d}^{\rm sc} |\Omega| \Lambda^{\gamma+d/2}
	\qquad\text{as}\ \Lambda\to\infty \,.
\end{equation}
Here, and in what follows, $L_{\gamma,d}^{\rm sc}$ denotes the so-called semiclassical constant
\begin{equation*}
	L_{\gamma,d}^{\rm sc} = \frac{\Gamma(1+\gamma)}{(4\pi)^{d/2}\Gamma(1+\gamma+d/2)}\,,
\end{equation*}
and we use the notation $x_\limpm = \frac{1}{2}(|x|\pm x)$. For $\gamma=0$, $\Tr (-\Delta_\Omega^\sharp -\Lambda)_\limminus^\gamma$ is interpreted as the number of eigenvalues $<\Lambda$. In the Neumann case the validity of the asymptotics \eqref{eq:weyl} requires some mild conditions on $\Omega$. We refer to \cite[Sections 3.2 and 3.3]{FrankLaptevWeidl} for a proof and further background.

It is remarkable that for $\gamma\geq 1$ the asymptotics \eqref{eq:weyl} are accompanied by uniform inequalities
\begin{equation}
	\label{eq:blyk}
	\Tr (-\Delta_\Omega^{\rm D} -\Lambda)_\limminus^\gamma \leq L_{\gamma,d}^{\rm sc} |\Omega| \Lambda^{\gamma+d/2} \leq \Tr (-\Delta_\Omega^{\rm N} -\Lambda)_\limminus^\gamma
	\qquad\text{for all}\ \Lambda\geq 0 \,.
\end{equation}
This was shown by Berezin \cite{Berezin}, Li and Yau \cite{LiYau_83} and by Kr\"oger \cite{Kroger_92} (see also \cite{Laptev_97}), respectively. The famous P\'olya conjecture states that these uniform inequalities are valid for all $\gamma\geq 0$. 

There is substantial literature on improvements of the Berezin--Li--Yau and Kr\"oger inequalities, of which we cite \cite{Freericks_etal_02,Melas_03,Ueltschi_04,LiTang_06,Weidl,KoVuWe,GeLaWe,KoWe,LarsonPAMS,HarrellStubbe_18,LarsonJST,Harrell_etal_21,FrankLarson_Convex24}. In most of these works, the respective inequalities are improved under the assumption that $\Omega$ belongs to a restricted class of open sets of finite measure. Typical assumptions are that $\Omega$ has a finite moment of inertia, that it is bounded, or that it admits a Hardy inequality. However, in many applications it is an important aspect of the Berezin--Li--Yau and Kr\"oger inequalities that they are valid for \emph{any} open subset $\Omega \subset \R^d$ of finite measure (which is necessary to state the inequalities). In this paper our main result is that improved versions of the Berezin--Li--Yau and Kr\"oger inequalities hold under the same minimal assumption, i.e., that $\Omega$ is an open set of finite measure.

Closest to this study, both from a technical and a conceptual point of view, are the work \cite{LiTang_06} by Li and Tang (where $\Omega$ is assumed to be bounded) and the work \cite{FrankLarson_Convex24} by two of us (where $\Omega$ is assumed to have finite width). In particular, in the latter paper it is shown that there are constants $c,c'>0$ such that for every open set $\Omega\subset\R^d$ of finite measure and finite width, one has
\begin{align}\label{eq: width bound D}
	\Tr (-\Delta_\Omega^{\rm D} -\Lambda)_\limminus \leq L_{1,d}^{\rm sc} |\Omega| \Lambda^{1+d/2} \left( 1 - c e^{-c' \sqrt{\Lambda} w_\Omega} \right)
\end{align}
and
\begin{align}\label{eq: width bound N}
	\Tr (-\Delta_\Omega^{\rm N} -\Lambda)_\limminus \geq L_{1,d}^{\rm sc} |\Omega| \Lambda^{1+d/2} \left( 1 + c e^{-c' \sqrt{\Lambda} w_\Omega} \right),
\end{align}
where $w_\Omega:= \inf_{\omega\in\Sph^{d-1}} (\sup_{x\in\Omega} \omega\cdot x - \inf_{x\in\Omega} \omega\cdot x)$ is the width of $\Omega$.

Our goal in this paper is three-fold. First, we extend the improved bounds from \cite{FrankLarson_Convex24} to arbitrary open sets of finite measure, without any additional assumptions. Second, we will extend the improved bounds to Riesz exponents $\gamma>1$. Third, we will prove corresponding improved bounds in the presence of a constant magnetic field in two dimensions. 

Before stating our bounds let us comment on the relative improvement in our bounds in comparison to the Berezin--Li--Yau and Kr\"oger inequalities. As in~\eqref{eq: width bound D} and \eqref{eq: width bound N}, our relative remainder is exponentially small in the limit $\Lambda\to \infty$. This might appear rather disappointing as two-term semiclassical asymptotics suggest a relative improvement of order $1/\sqrt{\Lambda}$ in this regime (at least for sufficiently regular $\Omega$) and indeed some of the above mentioned improvements capture at least an inverse power-like dependence on $\Lambda$. However, both~\eqref{eq: width bound D} and \eqref{eq: width bound N} provide a substantial improvement in the intermediate spectral regime $\sqrt{\Lambda}w_\Omega \lesssim 1$ and a corresponding statement holds for the results proved in this paper. In fact, the main application of~\eqref{eq: width bound D} and \eqref{eq: width bound N} in \cite{FrankLarson_Convex24} was to show that the validity of stronger improvements of the Berezin--Li--Yau and Kr\"oger inequalities which had been proved in the regime of large $\Lambda$ (by using semiclassical techniques as, for example, in \cite{FrankLarson_Invent25}) could be extended to the entire range $\Lambda \geq 0$. We believe that the results obtained here can provide a useful tool in the regime when $\Omega$ has non-trivial geometry on the natural length-scale $1/\sqrt{\Lambda}$, in which case semiclassical techniques are not applicable.

We begin with the non-magnetic case, where we prove the following.

\begin{theorem}
	\label{thm: main Laplace measure}
	For any $d\geq 1$ there are constants $c,c'>0$ such that for any open set $\Omega\subset\R^d$ of finite measure, any $\Lambda>0$ and any $\gamma\geq 1$ one has
	$$
	\Tr(-\Delta_\Omega^{\rm D} - \Lambda)_\limminus^\gamma \leq L_{\gamma,d}^{\rm sc}\, |\Omega| \, \Lambda^{\gamma+d/2} \left(1 - c e^{-c' \sqrt{\Lambda} \lvert \Omega \rvert^{1/d}} \right)
	$$
	and
	$$
	\Tr(-\Delta_\Omega^{\rm N} - \Lambda)_\limminus^\gamma \geq L_{\gamma,d}^{\rm sc}\, |\Omega| \, \Lambda^{\gamma+d/2} \left(1 + c e^{-c' \sqrt{\Lambda} \lvert \Omega \rvert^{1/d}} \right).
	$$
\end{theorem}

In fact, this theorem is a consequence of a more general result, Theorem \ref{thm: main Laplace inradius}, which we describe later in this introduction.

Let us turn to the case of a constant magnetic field in dimension $d=2$. Let $B>0$ and $A(x) := \frac{B}{2}\bigl(\begin{smallmatrix}x_2\\ -x_1\end{smallmatrix}\bigr)$ for $x\in\R^2$. Let $\Omega\subset\R^2$ be an open set and let $H_\Omega^{\rm D}$ and $H_\Omega^{\rm N}$ be the Dirichlet and Neumann realizations in $L^2(\Omega)$ of the Landau Hamiltonian 
$$
(-i\partial_{x_1} +A_1(x) )^2 + (-i\partial_{x_2} +A_2(x))^2 \,.
$$
Specifically, these operators are defined via the quadratic form (see \cite[Section 3.1]{FrankLaptevWeidl})
\begin{equation*}
 	u \mapsto \int_\Omega |\nabla u(x) +iA(x)u(x)|^2\,dx \,.
\end{equation*}
with form domains $H^1_{A,0}(\Omega)$ and $H^1_A(\Omega)$, respectively. Here $H^1_A(\Omega)$ consists of all functions in $L^2(\Omega)$ such that the distribution $(\nabla+iA)u$ is square-integrable in $\Omega$, and $H^1_{A,0}(\Omega)$ is the closure of $C_c^\infty(\Omega)$ in $H^1_A(\Omega)$.

We are interested in bounds on
$$
\Tr (H_\Omega^\sharp - \Lambda)_\limminus^\gamma
$$
that correspond to the above-mentioned strengthened versions of the Berezin--Li--Yau and Kr\"oger inequalities for the Laplace operator without magnetic field. The magnetic Berezin--Li--Yau and Kr\"oger inequalities were shown by Erd\H{o}s--Loss--Vougalter \cite{ErdosLossVougalter} and the first author~\cite{Frank_Semigroup}, respectively, and state that
$$
\Tr(H_\Omega^{\rm D} - \Lambda)_\limminus^\gamma \leq |\Omega| \, \frac{B}{2\pi} \sum_{k=1}^\infty (B(2k-1)-\Lambda)_\limminus^\gamma
\qquad\text{for all}\ \Lambda>0
$$
and
$$
\Tr(H_\Omega^{\rm N} - \Lambda)_\limminus^\gamma \geq |\Omega| \, \frac{B}{2\pi} \sum_{k=1}^\infty (B(2k-1)-\Lambda)_\limminus^\gamma
\qquad\text{for all}\ \Lambda>0 \,. 
$$
Note that here, instead of the quantity $L_{\gamma,2}^{\rm sc} \Lambda^{\gamma+1}$, the quantity $\frac{B}{2\pi} \sum_{k=1}^\infty (B(2k-1)-\Lambda)_\limminus^\gamma$ appears. This is crucial in applications where both $\Lambda$ and $B$ are considered as asymptotic parameters; see, for instance, \cite{CdV,LiSoYn1,LiSoYn2,So,FrLoWe,Fo}.

Our second main result is a strengthened version of the magnetic Berezin--Li--Yau and Kr\"oger inequalities.

\begin{theorem}
	\label{thm: main Landau measure}
	There are constants $c,c'>0$ such that for any open set $\Omega\subset R^2$ of finite measure, any $B>0$, any $\Lambda>0$ and any $\gamma\geq 1$ one has
	$$
	\Tr(H_\Omega^{\rm D} - \Lambda)_\limminus^\gamma \leq |\Omega| \, \frac{B}{2\pi} \sum_{k=1}^\infty (B(2k-1)-\Lambda)_\limminus^\gamma \left(1 - c e^{-c'(\sqrt{\Lambda \lvert \Omega \rvert} + B \lvert \Omega \rvert)} \right)
	$$
	and
	$$
	\Tr(H_\Omega^{\rm N} - \Lambda)_\limminus^\gamma \geq |\Omega| \, \frac{B}{2\pi} \sum_{k=1}^\infty (B(2k-1)-\Lambda)_\limminus^\gamma \left(1 + c e^{-c'(\sqrt{\Lambda \lvert \Omega \rvert} + B \lvert \Omega \rvert)} \right).
	$$
\end{theorem}

Again, this theorem is a consequence of a more general theorem that we describe momentarily.

These promised stronger results are stated in terms of a `regularized inradius'. Define for $\theta\in (0, 1)$ the regularized inradius of a measurable set $\Omega \subset \R^d$ by
\begin{equation}
	\rho_\theta(\Omega) := \inf\biggl\{\rho>0: \sup_{x\in \R^d} \frac{|\Omega \cap B_\rho(x)|}{|B_\rho(x)|}\leq \theta\biggr\}\,.
\end{equation}
Here, $B_\rho(x)$ denotes the open ball of radius $\rho$ centered at $x\in\R^d$. The relevance of the regularized inradius $\rho_\theta(\Omega)$ was emphasized by Lieb \cite{Lieb} in his lower bound on the smallest Dirichlet eigenvalue of the Laplacian. 

\begin{remark}\label{inradiusremarks}
	Let us summarize some facts about $\rho_\theta(\Omega)$.
	\begin{enumerate}
		\item[(a)] For any $\theta \in (0, 1)$ we have the bound
		$$
		\rho_\theta(\Omega) \leq \Bigl(\frac{|\Omega|}{\theta |B_1|}\Bigr)^{1/d}
		$$
		Indeed, for any $x \in \R^d$ and $\rho \geq \rho_0 := (\frac{|\Omega|}{\theta|B_1|})^{1/d}$ it holds that
		\begin{equation*}
			|\Omega \cap B_\rho(x)| \leq |\Omega| = \theta|B_{\rho_0}(x)| \leq \theta |B_\rho(x)|\,.
		\end{equation*}
		
		\item[(b)] We claim that for any $d\in\N$ and any $\theta\in(0,1)$ there is a $C$ such that for any set $\Omega\subset\R^d$ of finite width one has
		$$
		\rho_\theta(\Omega)\leq C\, w_\Omega
		$$
		Indeed, for any $\rho> w_\Omega/2, x\in \R^d$ it holds that
		\begin{align*}
			|\Omega \cap B_\rho(x)| 
			&\leq 
			\sup_{z \in \R^d}|\{y=(y_1, \ldots, y_d) \in \R^d: |y_d|<w_\Omega/2\}\cap B_\rho(z)| \\
			&= 
			|B_\rho(x)|\sup_{z \in \R^d}\frac{|\{y=(y_1, \ldots, y_d) \in \R^d: |y_d|<\frac{w_\Omega}{2\rho}\}\cap B_1(z)|}{|B_1|}\,.
		\end{align*}
		Since the supremum in the right-hand side tends to $0$ as $w_\Omega/(2\rho) \to 0$, for any given $\theta>0$ there is a $C$ such that it is $\leq\theta$ when $w_\Omega/(2\rho)\leq (2C)^{-1}$.
		
		\item[(c)] If $\rho_{\rm in}(\Omega):= \sup_{x\in\Omega} \dist(x,\partial\Omega)$ denotes the inradius of $\Omega$, then
        $$
        \rho_\theta(\Omega) \geq \rho_{\rm in}(\Omega)
        \qquad\text{for all}\ \theta\in(0,1) \,.
        $$
        This follows from the fact that $\sup_{x\in\R^d} |\Omega\cap B_\rho(x)|/|B_\rho(x)| = 1$ for $\rho\in(0,\rho_{\rm in}]$.
        
        \item[(d)] From the Lebesgue differentiation theorem $\lim_{\rho \to 0} \frac{|\Omega \cap B_\rho(x)|}{|B_\rho(x)|} = \1_\Omega(x)$ for almost every $x\in \R^d$. Therefore, $\rho_\theta(\Omega)=0$ if and only if $|\Omega|=0$.
		
		\item[(e)] If $\rho_\theta(\Omega)<\infty$, then for all $x \in \R^d$ it holds that
		$$
		|\Omega \cap B_{\rho_\theta(\Omega)}(x)| \leq \theta |B_{\rho_\theta(\Omega)}(x)|\,.
		$$ 
		Indeed, there is a sequence $\{\rho_k\}_{k\geq 1} \subset [\rho_\theta(\Omega), \infty)$ such that $\lim_{k\to \infty}\rho_k=\rho_\theta(\Omega)$ and $|\Omega \cap B_{\rho}(x)| \leq \theta |B_{\rho}(x)|$ for each $k\geq 1$ and all $x \in \R^d$. Since $\Omega \cap B_{\rho_\theta(\Omega)}(x)\subseteq \Omega \cap B_{\rho_k}(x)$ we deduce that
		\begin{equation}\label{eq: density estimate rho}
			|\Omega \cap B_{\rho_\theta(\Omega)}(x)| \leq |\Omega \cap B_{\rho_k}(x)| \leq \theta |B_{\rho_k}(x)| \quad \mbox{for all } k\geq 1, x\in \R^d\,.
		\end{equation}
		The claimed inequality follows by taking the limit $k \to \infty$. Moreover, for every $\epsilon \in (0, 1]$ there is an $x_\epsilon \in \R^d$ such that
        $$
            |\Omega \cap B_{\rho_\theta(\Omega)}(x_\epsilon)|\geq \theta(1-\epsilon)|B_{\rho_\theta(\Omega)}(x_\epsilon)|\,.
        $$
        Indeed, for each $\rho<\rho_\theta(\Omega)$ there is a $x'_\rho\in \R^d$ so that $|\Omega \cap B_{\rho}(x'_\rho)|\geq \theta|B_{\rho}(x'_\rho)|$, by inclusion and choosing $\rho \in [(1-\epsilon)^{1/d}\rho_\theta(\Omega), \rho_\theta(\Omega))$ we find
        $$
            |\Omega \cap B_{\rho_\theta(\Omega)}(x'_\rho)| \geq |\Omega \cap B_{\rho}(x'_\rho)| \geq \theta|B_\rho(x'_\rho)|\geq \theta(1-\epsilon)|B_{\rho_\theta(\Omega)}(x'_\rho)|
        $$
        which is the claimed inequality with $x_\epsilon=x'_\rho$. If $|\Omega|<\infty$ one can extend the above inequality to $\epsilon=0$ by arguing that non-compactness of the set $\{x_\epsilon\}_{\epsilon\in (0, 1)}$ would contradict the finiteness of $|\Omega|$.
	\end{enumerate}
\end{remark}

We now state the stronger versions of Theorems \ref{thm: main Laplace measure} and \ref{thm: main Landau measure}.

\begin{theorem}
    \label{thm: main Laplace inradius}
    For any $d\geq 1$ and $\theta\in(0,1)$ there are constants $c,c'>0$ such that for any open set $\Omega\subset\R^d$ of finite measure, any $\Lambda>0$ and any $\gamma\geq 1$ one has
    $$
	\Tr(-\Delta_\Omega^{\rm D} - \Lambda)_\limminus^\gamma \leq L_{\gamma,d}^{\rm sc}\, |\Omega| \, \Lambda^{\gamma+d/2} \left(1 - c e^{-c' \sqrt{\Lambda} \rho_\theta(\Omega)} \right)
	$$
	and
	$$
	\Tr(-\Delta_\Omega^{\rm N} - \Lambda)_\limminus^\gamma \geq L_{\gamma,d}^{\rm sc}\, |\Omega| \, \Lambda^{\gamma+d/2} \left(1 + c e^{-c' \sqrt{\Lambda} \rho_\theta(\Omega)} \right).
	$$
\end{theorem}
        
Clearly, in view of Remark \ref{inradiusremarks} (a), Theorem \ref{thm: main Laplace measure} is an immediate consequence of Theorem \ref{thm: main Laplace inradius}, taking any fixed value of $\theta$. Similarly, the improved bound in \cite[Proposition 2.1]{FrankLarson_Convex24} is a consequence of Theorem \ref{thm: main Laplace measure} and Remark \ref{inradiusremarks} (b), except for the fact that in \cite{FrankLarson_Convex24} we also proved that the constants $c,c'$ can be chosen independently of $d$.

In the magnetic case we have the corresponding result.

\begin{theorem}
	\label{thm: main Landau inradius}
	For any $\theta \in (0, 1)$ there are constants $c,c'>0$ such that for any open set $\Omega\subset\R^2$ with $|\Omega|<\infty$, any $B>0$, any $\Lambda>0$ and any $\gamma\geq 1$ one has
	$$
	\Tr(H_\Omega^{\rm D} - \Lambda)_\limminus^\gamma \leq |\Omega| \, \frac{B}{2\pi} \sum_{k=1}^\infty (B(2k-1)-\Lambda)_\limminus^\gamma \left(1 - c e^{-c'(\rho_\theta(\Omega)\sqrt{\Lambda} + \rho_\theta(\Omega)^2B )} \right)
	$$
	and
	$$
	\Tr(H_\Omega^{\rm N} - \Lambda)_\limminus^\gamma \geq |\Omega| \, \frac{B}{2\pi} \sum_{k=1}^\infty (B(2k-1)-\Lambda)_\limminus^\gamma \left(1 + c e^{-c'(\rho_\theta(\Omega)\sqrt{\Lambda} + \rho_\theta(\Omega)^2B )} \right).
	$$
\end{theorem}

Again, in view of Remark \ref{inradiusremarks} (a), Theorem \ref{thm: main Landau measure} is an immediate consequence of Theorem~\ref{thm: main Landau inradius}, taking any fixed value of $\theta$.

Let us explain the basic idea of the proof of Theorems \ref{thm: main Laplace inradius} and \ref{thm: main Landau inradius}. Following the classical proofs of the Berezin--Li--Yau and Kr\"oger inequalities, as well as their magnetic analogues, we arrive at terms of the form
$$
\| \1(L>\Lambda) J \psi \|^2 \,,
$$
where $L$ is either the Laplacian $-\Delta$ in $L^2(\R^d)$ or the Landau Hamiltonian $H$ in $L^2(\R^2)$, where $\psi$ is an eigenfunction of either the Dirichlet or the Neumann restriction of these operators to a set $\Omega$ of finite measure and where $J$ denotes the extension by zero of a function defined on $\Omega$ to $\R^d$. In the standard proofs, the above term is dropped since it is nonnegative.

This is wasteful, since the term $\| \1(L>\Lambda) J \psi \|^2$ is never equal to zero. In fact, a version of the uncertainty principle says that a function $\psi$ cannot be simultaneously localized in a set $\Omega$ of finite measure and supported in the spectral subspace $\ran\1(L\leq\Lambda)$. Thus, our task is to find a quantitative version of the uncertainty principle to obtain a positive lower bound on the term $\| \1(L>\Lambda) J \psi \|^2$. As we will show, such lower bounds can be deduced from inequalities known as `spectral inequalities' in the control theory community and as the `Logvinenko--Sereda theorem' in harmonic analysis. For these inequalities the notion of a thick set is relevant and this is where our generalized inradius $\rho_\theta(\Omega)$ appears naturally.

The fact that our arguments in the non-magnetic and in the magnetic case are rather parallel underlines the general mechanism that leads from quantitative versions of the uncertainty principle to improved semiclassical eigenvalue bounds. Our strategy might be applicable in other settings as well, for instance in the setting of homogeneous spaces \cite{St} or of the Heisenberg group \cite{HaLa}.

This paper is organized as follows. Section \ref{sec:uncertainty} is devoted to quantitative versions of the uncertainty principle. The corresponding main results are stated in Subsection~\ref{sec:keybound} and proved in Subsection \ref{sec:keyboundproof}, after having recalled some necessary spectral inequalities in Subsection \ref{sec:specineq}. Section \ref{sec:gamma=} and, in particular, Subsection \ref{sec:mainproof} are devoted to the proof of our main results, Theorems \ref{thm: main Laplace inradius} and \ref{thm: main Landau inradius}, in the case $\gamma=1$. In order to streamline the proofs in the magnetic and in the non-magnetic case, we present the identities underlying the Berezin--Li--Yau and Kr\"oger arguments in an abstract setting in Subsection \ref{sec:abstract}. Finally, in Section \ref{sec:gamma>} we deduce the case $\gamma>1$ in Theorems \ref{thm: main Laplace inradius} and~\ref{thm: main Landau inradius} from the case $\gamma=1$.

%%%%%%%%%%%%%%%%%%%%%%%%%%%%%%%

\section{An uncertainty principle}\label{sec:uncertainty}

\subsection{The key bound}\label{sec:keybound}
The goal of this section is to prove bounds that quantify the fact that if a function $f\not\equiv 0$ is supported on a set of finite measure, then its `high-energy' part $\1(-\Delta>\Lambda) f$ cannot vanish identically. The same is true when $-\Delta$ is replaced by the Landau Hamiltonian. By quantifying this fact we mean giving a lower bound on the norm of $\1(-\Delta>\Lambda) f$ in terms of the norm of $f$, and it is in these bounds that the regularized inradius $\rho_\theta(\Omega)$ enters.

The Fourier transform is denoted by
$$
\widehat f(\xi) = (2\pi)^{-d/2} \int_{\R^d} e^{-i\xi\cdot x} f(x)\,dx \,,
\qquad \xi\in\R^d \,.
$$

\begin{theorem}\label{thm: uncertainty principle Laplacian}
    Let $d\geq 1$. For any $\theta \in (0, 1)$ there are constants $c,c'>0$ such that for any measurable set $\Omega\subset\R^d$ with $\rho_\theta(\Omega)<\infty$, any $f\in L^2(\R^d)$ and any $\Lambda>0$, one has
\begin{align*}
    \int_{|\xi|^2>\Lambda} |\widehat{\1_\Omega f}(\xi) |^2\,d\xi \geq c \, e^{-c' \rho_\theta(\Omega)\sqrt{\Lambda}} \,\lVert \1_\Omega f \rVert^2 \,.
\end{align*}
\end{theorem}

We now state a corresponding result in the magnetic case. As is well known, the Landau Hamiltonian $H=(-i\partial_{x_1} + \tfrac B2 x_2 )^2 + (-i\partial_{x_2} - \tfrac B2 x_1)^2$ in $L^2(\R^2)$ can be diagonalized explicitly. Its spectrum consists of the eigenvalues $B(2k-1)$, $k\in\N=\{1,2,3,\ldots\}$, each of which has infinite multiplicity. We denote by $\Pi_k$ the spectral projection of $H$ corresponding to the eigenvalue $B(2k-1)$.

We shall show the following theorem.

\begin{theorem} \label{thm: uncertainty principle Landau}
For any $\theta \in (0, 1)$ there are constants $c,c'>0$ such that for any measurable set $\Omega\subset\R^2$ with $\rho_\theta(\Omega)<\infty$, any $f \in L^2(\R^2)$, any $B>0$ and any $\Lambda>0$, one has
\begin{align*}
\sum_{B(2k-1) > \Lambda} \|\Pi_k \1_\Omega f\|^2 \geq c \, e^{-c'( \rho_\theta(\Omega)\sqrt{\Lambda} +  \rho_\theta(\Omega)^2B)} \, \lVert \1_\Omega f \rVert^2 \,.
\end{align*}
\end{theorem}

The remainder of this section is devoted to the proof of these two theorems. They will be deduced from results that are known in the literature as `spectral inequalities'.

%%%%%%%%%%%%%%%%%%%%%%%%%%

\subsection{Known spectral inequalities}\label{sec:specineq}

For $\rho>0, \kappa \in (0, 1)$ we say that a set $S \subset \R^d$ is $(\rho, \kappa)$-thick if it is measurable and satisfies
\begin{equation*}
	|S\cap B_\rho(x)| \geq \kappa |B_\rho(x)| \quad \mbox{for all }x\in \R^d \,.
\end{equation*}

\begin{theorem} \label{thm: spectral ineq Laplace}	For any $d\in\N$ and $\kappa\in(0,1)$ there are constants $c,c'>0$ such that the following holds. Let $S \subset \R^d$ be an $(\rho, \kappa)$-thick set and let $\Lambda>0$. Then for any $g\in \ran\1(-\Delta\leq\Lambda)$ we have 
\begin{align*}
\lVert  \1_S  g \rVert^2 \ge c \, e^{-c' \rho \sqrt{\Lambda}} \, \lVert g \rVert^2 \, .
\end{align*}
\end{theorem}

This result is due to Kacnel'son \cite{Ka} and Logvinenko and Sereda \cite{LoSe}, who provided a complete proof for $d=1$. The necessary modifications for $d\geq 2$ can be found in \cite[Lemma 2.1]{Wang_etal_19}. We also refer to the latter paper for a further discussion of its history. In particular, we mention the work of Kovrijkine \cite{Ko}, where the dependence of the constants $c$ and $c'$ on $\kappa$ is quantified.

In the recent paper \cite{PfeifferTaufer}, M.\ T\"aufer and the third named author extended Theorem \ref{thm: spectral ineq Laplace} to the case of a constant magnetic field in two dimensions.

\begin{theorem} \label{thm: spectral ineq Landau}
	For any $\kappa\in(0,1)$ there are constants $c,c'>0$ such that the following holds. Let $S \subset \R^2$ be an $(\rho, \kappa)$-thick set and let $B,\Lambda>0$. Then for any $g\in\ran(H\leq\Lambda)$ we have 
\begin{align*}
\lVert  \1_S  g \rVert^2 \ge c \, e^{-c' (\rho \sqrt{\Lambda} + \rho^2 B)} \, \lVert g \rVert^2 \, .
\end{align*}
\end{theorem}

In \cite{PfeifferTaufer} a more general version of thickness is used in terms of rectangles with side lengths $\ell_1, \ell_2$ playing the role of the disks $B_\rho(x)$ in the definition used here. However, any set that is $(\rho, \kappa)$-thick according to the definition used here can easily be verified to be thick in the sense of \cite{PfeifferTaufer} with respect to the square $Q_\rho$ with side length $2\rho$, for which we have $B_\rho(x) \subset Q_\rho(x)= x+(-\rho, \rho)^2$ and $|B_\rho(x)| = (\pi/4)|Q_\rho|$.

%%%%%%%%%%%%%%%%%%%%%%%%%%%%%%%

\subsection{Proof of Theorems \ref{thm: uncertainty principle Laplacian} and \ref{thm: uncertainty principle Landau}}\label{sec:keyboundproof}

Theorems~\ref{thm: spectral ineq Laplace} and \ref{thm: spectral ineq Landau} deal with functions from the `low energy' subspaces $\ran\1(-\Delta\leq\Lambda)$ and $\ran\1(H\leq\Lambda)$, while Theorems \ref{thm: uncertainty principle Laplacian} and \ref{thm: uncertainty principle Landau} involve projections onto the `high energy' subspaces $\ran\1(-\Delta>\Lambda)$ and $\ran\1(H>\Lambda)$. The following lemma will be useful to pass from the first to the second setting. We will apply it with $P = \1(-\Delta>\Lambda)$ or $P=\1(H>\Lambda)$ and with $Q$ defined through multiplication with $\1_\Omega$.

\begin{lemma}\label{lem: projection lemma}
Let $P, Q$ be selfadjoint projections in a Hilbert space $\mathcal{H}$. Then, for any $g \in \mathcal{H}$, it holds that
\begin{align*}
\| PQ g \|_{\mathcal{H}}^2\|(\1-P)Q g\|_{\mathcal{H}}^2 \ge \|Q g \|^2_{\mathcal{H}} \| (\1-Q)(\1-P) Qg \|_{\mathcal{H}}^2 \,.
\end{align*}
In particular, if $g \in \ran(Q)\setminus \ran(P)$, then
\begin{align*}
\| Pg \|_{\mathcal{H}}^2 \ge \|g \|^2_{\mathcal{H}} \frac { \| (\1-Q)(\1-P) g \|_{\mathcal{H}}^2 } {\|(\1-P) g\|_{\mathcal{H}}^2} \,.
\end{align*}
\end{lemma}

\begin{proof}
Using the fact that $P, Q$ are selfadjoint projections and applying the Cauchy--Schwarz inequality, one finds
\begin{equation}\label{eq: CS estimate}
\begin{aligned}
\| (\1-P) Qg \|_{\mathcal{H}}^4 &= \langle Qg, (\1-P) Qg \rangle_{\mathcal{H}}^2 \\
&= \langle Qg, Q  (\1-P) Qg \rangle_{\mathcal{H}}^2 \\
&\le \| Qg \|_{\mathcal{H}}^2 \| Q  (\1-P) Qg  \|_{\mathcal{H}}^2\,.
\end{aligned}
\end{equation}

Recall that if $P'$ is a selfadjoint projection in $\mathcal{H}$, then 
\begin{equation}\label{eq: pythagorean thm proj}
\| h \|_{\mathcal{H}}^2 = \| P' h \|_{\mathcal{H}}^2+ \| (\1-P') h \|_{\mathcal{H}}^2\quad \mbox{for all } h \in \mathcal{H}\,.
\end{equation}
In particular, it holds that
\begin{align*}
\| (\1-P) Qg \|_{\mathcal{H}}^2 &= \|Qg\|_{\mathcal{H}}^2-\|PQg\|_{\mathcal{H}}^2\,, \\
\| Q  (\1-P) Qg  \|_{\mathcal{H}}^2 &= \| (\1-P) Qg  \|_{\mathcal{H}}^2-\| (\1-Q)  (\1-P) Qg  \|_{\mathcal{H}}^2\,.
\end{align*}
Thus \eqref{eq: CS estimate} is equivalent to
\begin{equation*}
	\| (\1-P) Qg \|_{\mathcal{H}}^2\bigl(\|Qg\|_{\mathcal{H}}^2-\|PQg\|_{\mathcal{H}}^2\bigr) 
	\le \| Qg \|_{\mathcal{H}}^2 \bigl(\| (\1-P) Qg  \|_{\mathcal{H}}^2-\| (\1-Q)  (\1-P) Qg  \|_{\mathcal{H}}^2\bigr)\,,
\end{equation*}
which, after rearranging the terms, yields the first estimate in the lemma.

The second estimate follows by noting that for $g\in \ran(Q)\setminus \ran(P)$ we have $Qg=g$ and $\|(\1-P)g\|_{\mathcal{H}}\neq 0$ since $\ran(P)=\ker(\1-P)$.
\end{proof}

We can now prove Theorems \ref{thm: uncertainty principle Laplacian} and \ref{thm: uncertainty principle Landau}. As a preliminary remark, let us find a connection between the regularized inradius of a set and the thickness properties of its complement.

We claim that, if $\rho_\theta(\Omega)<\infty$, then $\Omega^\complement$ is $(\rho_\theta(\Omega), 1-\theta)$-thick. Indeed, for any $x \in \R^d$ it follows from~\eqref{eq: density estimate rho} that
\begin{align*}
	|\Omega^\complement  \cap B_{\rho_\theta(\Omega)} (x)| = |B_{\rho_\theta(\Omega)} (x)|- |\Omega \cap B_{\rho_\theta(\Omega)}| \geq  (1-\theta)|B_{\rho_\theta(\Omega)} (x)|\,.
\end{align*}
Conversely, if $\Omega^\complement$ is $(\rho, \kappa)$-thick then $\rho_{1-\kappa}(\Omega)\leq \rho$; indeed
\begin{equation*}
 	|\Omega\cap B_{\rho} (x)| = |B_{\rho} (x)|- |\Omega^\complement \cap B_{\rho}(x)| \leq  (1-\kappa)|B_{\rho} (x)|\quad \mbox{for all } x\in \R^d\,.
\end{equation*} 

The proofs of Theorems \ref{thm: uncertainty principle Laplacian} and \ref{thm: uncertainty principle Landau} are completely analogous and we only give the details of the second one.

\begin{proof}[Proof of Theorem \ref{thm: uncertainty principle Landau}]
Let $P, Q$ be the selfadjoint projections on $L^2(\R^2)$ given by $P := \1(H>\Lambda)$ and $Q$ defined by multiplication by $\1_\Omega$. Note that for these projections $\1-P=\1(H\leq\Lambda)$ and $\1-Q$ is given by multiplication by $\1_{\Omega^\complement}$. 

Fix $f \in L^2(\R^2)$. If $\1_\Omega f \in \ran(P)$ then 
\begin{equation*}
	\sum_{B(2k-1)>\Lambda} \|\Pi_k \1_\Omega f\|^2 = \| P \1_\Omega f \|^2 = \|\1_\Omega f\|^2
\end{equation*}
and the claimed inequality holds as soon as $c\leq 1$. If $\1_\Omega f \notin \ran(P)$ we can instead use the second bound in Lemma~\ref{lem: projection lemma} applied to $g:=\1_\Omega f \in \ran(Q)\setminus \ran(P)$ to deduce that
\begin{align*}
	\sum_{B(2k-1)>\Lambda} \|\Pi_k \1_\Omega f\|^2 = \| P \1_\Omega f \|^2  
	\ge \| \1_\Omega f \|^2 \frac { \| \1_{\Omega^\complement} (\1-P) \1_\Omega f\|^2 }{ \|(\1-P) \1_\Omega f \|^2} \, .
\end{align*}
We can now apply Theorem \ref{thm: spectral ineq Landau} to $g:= (\1-P) \1_\Omega f$ and conclude that 
\begin{align*}
\sum_{B(2k-1)>\Lambda} \|\Pi_k \1_\Omega f\|^2 \ge \| \1_\Omega f \|^2 c e^{- c' ( \rho_\theta(\Omega) \sqrt{\Lambda}+ \rho_\theta(\Omega)^2B)} \, ,
\end{align*}
with $c, c'$ as in Theorem~\ref{thm: spectral ineq Landau}. Combining the two cases proves Theorem~\ref{thm: uncertainty principle Landau}.
\end{proof}

%%%%%%%%%%%%%%%%%%%%%%%%%%%%%%%

\section{Proof of the improved semiclassical bounds}\label{sec:gamma=}

\subsection{Abstract formulas}\label{sec:abstract}

We begin by proving two somewhat abstract identities which will be at the core of our proofs of our main theorems. Although not written down in the general form we prove them here, these identities are central in the proofs of the Berezin--Li--Yau \cite{Berezin,LiYau_83} and Kr\"oger \cite{Kroger_92,Laptev_97} inequalities, as well as their magnetic analogues proved in \cite{ErdosLossVougalter} and \cite{Frank_Semigroup}. We believe that the abstract formulation presented here captures that the underlying principle is very general, and that unifying the proofs will be beneficial for future applications in other contexts. One specific advantage of our formulation is that it completely avoids distinguishing into cases depending on the nature of the spectrum of the involved operators.

\begin{lemma}\label{lem: abstract identities}
    Let $\mathcal H$ and $\hat{\mathcal H}$ be separable Hilbert spaces and $J:\mathcal H\to\hat{\mathcal H}$ an isometry. Let $L$ and $\hat L$ be nonnegative selfadjoint operators in $\mathcal H$ and $\hat{\mathcal H}$, respectively, with spectral measures $E$ and $\hat E$ and quadratic form domains $\mathcal{Q}$ and $\hat{\mathcal{Q}}$. Set $E^\bot = \1-E$ and $\hat E^\bot = \1-\hat E$ and let $\Lambda>0$.
    \begin{enumerate}
        \item\label{itm: Abstract Neu} If $J^* \hat{\mathcal Q}\subset \mathcal Q$ and $\hat E(\Lambda)J(L-\Lambda)J^*\hat E(\Lambda)$ is trace class, then
        $$
        \Tr (L-\Lambda)_\limminus = - \int (\lambda - \Lambda) \,d\Tr \hat E(\Lambda) J E(\lambda) J^* \hat E(\Lambda) + \mathcal R_< + \mathcal R_>
        $$
        where
        \begin{align*}
            \mathcal R_< & = \int (\lambda-\Lambda)_\limminus \, d \Tr \hat E(\Lambda)^\bot  J E(\lambda) J^* \hat E(\Lambda)^\bot \,,\\
            \mathcal R_> & = \int (\lambda-\Lambda)_\limplus \,d \Tr \hat E(\Lambda)  J E(\lambda) J^* \hat E(\Lambda) \,.
        \end{align*}
        \item\label{itm: Abstract Dir} If $L= J^* \hat L J$ in the sense of quadratic forms, then
        $$
        \Tr (L-\Lambda)_\limminus = \int (\lambda - \Lambda)_\limminus \,d\Tr J^* \hat E(\lambda) J - \mathcal R_<' - \mathcal R_>'
        $$
        where
        \begin{align*}
            \mathcal R_<' & := \int (\lambda-\Lambda)_\limplus \, d \Tr E(\Lambda)  J^* \hat E(\lambda) J E(\Lambda) \,,\\
            \mathcal R_>' & := \int (\lambda-\Lambda)_\limminus \, d \Tr E(\Lambda)^\bot  J^* \hat E(\lambda) J E(\Lambda)^\bot \,.
        \end{align*}
    \end{enumerate}
\end{lemma}

By the assumption $L=J^* \hat L J$ in the sense of quadratic forms we mean the following. Let $q$ and $\hat q$ be the quadratic forms of $L$ and $\hat L$, respectively. Then we assume $J \mathcal Q\subset \hat{\mathcal Q}$ and $\hat q[J\psi] = q[\psi]$ for all $\psi\in\mathcal Q$.

We have written the second identity in an elegant, but slightly ambiguous way because there could be cancellations of infinities, leaving the expressions undefined. The precise meaning is as follows. The identity in \eqref{itm: Abstract Dir} is to be understood with $\mathcal R_<'$ and $\mathcal R_>'$ added to both sides. Then the left side is a sum of three positive terms, while there is one positive term on the right side. Identity \eqref{itm: Abstract Dir} rewritten in this form holds in the sense of an equality of elements in $[0,\infty)\cup\{+\infty\}$, which is also the case for the identity in \eqref{itm: Abstract Neu}.

\begin{proof}
    We begin with the proof of \eqref{itm: Abstract Neu}. In the sense of bounded operators in $\hat{\mathcal H}$
    \begin{align*}
        J(L-\Lambda)_\limminus J^* 
            &= (\hat E(\Lambda)+\hat E(\Lambda)^\perp)J(L-\Lambda)_\limminus J^*(\hat E(\Lambda)+\hat E(\Lambda)^\perp)\\
            &=\hat E(\Lambda) J(L-\Lambda)_\limminus J^*\hat E(\Lambda) +\hat E(\Lambda)^\perp J(L-\Lambda)_\limminus J^*\hat E(\Lambda)^\perp\\
            &\quad +\hat E(\Lambda) J(L-\Lambda)_\limminus J^*\hat E(\Lambda)^\perp+\hat E(\Lambda)^\perp J(L-\Lambda)_\limminus J^*\hat E(\Lambda)\,.
    \end{align*}
    In the sense of quadratic forms on $\hat{\mathcal H}$ we can rewrite the first term in the right-hand side as
    \begin{align*}
        \hat E(\Lambda) J(L-\Lambda)_\limminus J^*\hat E(\Lambda) = -\hat E(\Lambda) J(L-\Lambda)J^*\hat E(\Lambda) + \hat E(\Lambda) J(L-\Lambda)_\limplus J^*\hat E(\Lambda)\,.
    \end{align*}
    By assumption $\hat E(\Lambda)J(L-\Lambda)J^*\hat E(\Lambda)$ is trace class and, therefore, bounded. Since the left-hand side is bounded, the same must be true for the second term in the right-hand side. Therefore, this identity holds in the sense of bounded operators on $\hat{\mathcal{H}}$.
    
    When combined we have shown that
    \begin{equation}\label{eq: op identity Neu}
    \begin{aligned}
        J(L-\Lambda)_\limminus J^*
            &=-\hat E(\Lambda) J(L-\Lambda)J^*\hat E(\Lambda) + \hat E(\Lambda) J(L-\Lambda)_\limplus J^*\hat E(\Lambda)\\
            &\quad +\hat E(\Lambda)^\perp J(L-\Lambda)_\limminus J^*\hat E(\Lambda)^\perp +\hat E(\Lambda) J(L-\Lambda)_\limminus J^*\hat E(\Lambda)^\perp\\
            &\quad
            +\hat E(\Lambda)^\perp J(L-\Lambda)_\limminus J^*\hat E(\Lambda)\,.
    \end{aligned}
    \end{equation}

    As we can write $\hat{\mathcal{H}}$ as the direct sum of $\ran(\hat E(\Lambda))$ and $\ran(\hat E(\Lambda)^\perp)$, there is a complete orthonormal set $\{\psi_k\}_{k\geq 1}\subset \hat{\mathcal{H}}$ such that, for each $k$, $\psi_k\in \hat{\mathcal{Q}}$ and one of $\hat E(\Lambda)\psi_k$ and $\hat E(\Lambda)^\perp \psi_k$ is zero. Thus for each $K\geq 1$ the identity in \eqref{eq: op identity Neu} implies
    \begin{align*}
        \sum_{k=1}^K\langle \psi_k, J(L-\Lambda)_\limminus J^*\psi_k\rangle_{\hat{\mathcal{H}}}
            &=-\sum_{k=1}^K\langle \psi_k,\hat E(\Lambda) J(L-\Lambda)J^*\hat E(\Lambda)\psi_k\rangle_{\hat{\mathcal{H}}} \\
            &\quad + \sum_{k=1}^K\langle \psi_k,\hat E(\Lambda) J(L-\Lambda)_\limplus J^*\hat E(\Lambda)\psi_k\rangle_{\hat{\mathcal{H}}}\\
            &\quad 
            +\sum_{k=1}^K\langle \psi_k,\hat E(\Lambda)^\perp J(L-\Lambda)_\limminus J^*\hat E(\Lambda)^\perp\psi_k\rangle_{\hat{\mathcal{H}}}\,.
    \end{align*}
    Note that all terms in each of the sums except for the first one in the right-hand side are nonnegative. However, the limit as $K\to\infty$ of the first sum on the right exists by the assumption that $\hat E(\Lambda) J(L-\Lambda)J^*\hat E(\Lambda)$ is trace class. Therefore, by taking the limit $K \to \infty$ we conclude that
    \begin{align*}
        \Tr(J(L-\Lambda)_\limminus J^*)
            &=-\Tr(\hat E(\Lambda) J(L-\Lambda)J^*\hat E(\Lambda)) \\
            &\quad + \Tr(\hat E(\Lambda) J(L-\Lambda)_\limplus J^*\hat E(\Lambda))\\
            &\quad 
            +\Tr(\hat E(\Lambda)^\perp J(L-\Lambda)_\limminus J^*\hat E(\Lambda)^\perp)\,.
    \end{align*}
    Here we used the definition of the trace of a nonnegative, bounded operator; see \cite[Section VI.6]{ReSi1}. Using the cyclicity of the trace, that $J^*J = \1$, and rewriting the traces in the right-hand side in terms of spectral measures proves the claimed identity.
    
    \medskip

    We turn now to the proof of \eqref{itm: Abstract Dir}. As bounded operators in $\mathcal H$ we have
    $$
    (L-\Lambda)_\limminus + E(\Lambda) L E(\Lambda) = \Lambda E(\Lambda) \,.
    $$
    Using $J^*J=\1$ and $L=J^* \hat L J$, we can rewrite this as
    $$
        (L-\Lambda)_\limminus + E(\Lambda) J^* \hat L J E(\Lambda) =  E(\Lambda) J^* \Lambda J E(\Lambda) \,,
    $$
    
    Furthermore, using $x = (x-\Lambda)_\limplus+\min\{x, \Lambda\}$ for $x\geq 0$, we deduce that
    \begin{align*}
        E(\Lambda) J^* \hat L J E(\Lambda) &= E(\Lambda) J^* ((\hat L-\Lambda)_\limplus+\min\{\hat L, \Lambda\}) J E(\Lambda)\\
        &= E(\Lambda) J^*(\hat L-\Lambda)_\limplus J E(\Lambda)+E(\Lambda) J^* \min\{\hat L, \Lambda\}J E(\Lambda)\,.
    \end{align*}
    Since the term on the left and the second term on the right are both bounded operators, the same must be true for the first term on the right. (The term on the left is bounded since it is equal to $E(\Lambda)L E(\Lambda)$.)

    Similarly, writing $\Lambda = (x-\Lambda)_\limminus+\min\{x, \Lambda\}$ for $x\geq 0$, we have that
    \begin{align*}
        E(\Lambda) J^* \Lambda J E(\Lambda) &= E(\Lambda) J^* ((\hat L-\Lambda)_\limminus+\min\{\hat L, \Lambda\}) J E(\Lambda)\\
        &= E(\Lambda) J^*(\hat L-\Lambda)_\limminus J E(\Lambda)+E(\Lambda) J^* \min\{\hat L, \Lambda\}J E(\Lambda)\,.
    \end{align*}

    Combining the above we have shown that as bounded operators on $\mathcal H$
    \begin{align*}
        & (L-\Lambda)_\limminus + E(\Lambda) J^*(\hat L-\Lambda)_\limplus J E(\Lambda)+E(\Lambda) J^* \min\{\hat L, \Lambda\}J E(\Lambda)\\
        &\quad =  E(\Lambda) J^*(\hat L-\Lambda)_\limminus J E(\Lambda)+E(\Lambda) J^* \min\{\hat L, \Lambda\}J E(\Lambda) \,.
    \end{align*}
    Since each of the terms are bounded operators, this is implies that
    \begin{align*}
        (L-\Lambda)_\limminus + E(\Lambda) J^*(\hat L-\Lambda)_\limplus J E(\Lambda)=  E(\Lambda) J^*(\hat L-\Lambda)_\limminus J E(\Lambda)\,.
    \end{align*}

    Note that the operator
    $$
        E(\Lambda)^\bot J^* (\hat L-\Lambda)_\limminus J E(\Lambda)^\bot+E(\Lambda) J^* (\hat L-\Lambda)_\limminus J E(\Lambda)^\bot+E(\Lambda)^\bot J^* (\hat L-\Lambda)_\limminus J E(\Lambda)
    $$
    is bounded, since $(\hat L-\Lambda)_\limminus$ is bounded. Adding this operator to both sides, we obtain that as sums of bounded operators on $\mathcal H$,
    \begin{align*}
        J^*(\hat L-\Lambda)_\limminus J &= (L-\Lambda)_\limminus + E(\Lambda) J^*(\hat L-\Lambda)_\limplus J E(\Lambda)+E(\Lambda)^\bot J^* (\hat L-\Lambda)_\limminus J E(\Lambda)^\bot\\
        &\quad +E(\Lambda) J^* (\hat L-\Lambda)_\limminus J E(\Lambda)^\bot+E(\Lambda)^\bot J^* (\hat L-\Lambda)_\limminus J E(\Lambda)\,.
    \end{align*}
    The proof can now be completed in the same way as for the first identity in the lemma.
\end{proof}

\begin{remark}
    It is quite possible that Lemma \ref{lem: abstract identities} remains valid without the assumption that $\mathcal H$ and $\hat{\mathcal H}$ are separable. This would probably involve a definition of the trace of a bounded nonnegative operator via nets. Since the separability assumption is satisfied in all our examples, we content ourselves with this case.
\end{remark}

%%%%%%%%%%%%%%%%%%%%%%%%%%%%%%%%%%%%%%

\subsection{Proof of main results}\label{sec:mainproof}

As mentioned above, the identities in Lemma~\ref{lem: abstract identities} can be used to prove the inequalities of Berezin--Li--Yau, Kr\"oger, Erd\H os--Loss--Vougalter and one of the authors. Indeed, as we shall see below, applying \eqref{itm: Abstract Dir} of Lemma~\ref{lem: abstract identities} with $\mathcal{H} = L^2(\Omega), \hat{\mathcal H}= L^2(\R^d)$, $L$ being either $-\Delta^{\rm D}_\Omega$ or $H_\Omega^{\rm D}$, $\hat L$ being $-\Delta$ or $H$, and $J\colon L^2(\Omega) \to L^2(\R^d)$ defined as extension by zero and estimating $\mathcal R_<'\geq 0, \mathcal R_>'\geq 0$ yields the Dirichlet version of these inequalities. Similarly, applying \eqref{itm: Abstract Neu} of Lemma~\ref{lem: abstract identities} with $\mathcal{H} = L^2(\Omega), \hat{\mathcal H}= L^2(\R^d)$, $L$ being either $-\Delta^{\rm N}_\Omega$ or $H_\Omega^{\rm N}$, $\hat L$ being $-\Delta$ or $H$, and $J\colon L^2(\Omega) \to L^2(\R^d)$ defined as extension by zero and estimating $\mathcal R_<\geq 0, \mathcal R_>\geq 0$ yields the Neumann version of these inequalities. To obtain our improved inequalities we shall utilize the uncertainty principles discussed in the first part of this paper to provide a positive lower bound for $\mathcal R_<$ and $\mathcal R_<'$, respectively. 

In each case we shall apply the uncertainty principles of Theorems~\ref{thm: uncertainty principle Laplacian} and \ref{thm: uncertainty principle Landau} with $f$ being an eigenfunction of the operator of interest (extended by zero to all of $\R^d$). As such it is not unreasonable to believe that the conclusions of the uncertainty principles could be strengthened in this restricted class of functions. However, under the minimal assumption that $|\Omega|<\infty$ it is difficult to control global properties of the eigenfunctions in any uniform manner. Therefore, strengthening the conclusions of Theorems~\ref{thm: uncertainty principle Laplacian} and \ref{thm: uncertainty principle Landau} for these eigenfunctions appears challenging without imposing further assumptions on $\Omega$.

\begin{proof}[Proof of Theorem~\ref{thm: main Laplace inradius} for $-\Delta_\Omega^{\rm D}, \gamma=1$]
    Let $\Omega\subset\R^d$ be an open set of finite measure. 

\medskip
    
    \emph{Step 1.} We recall that the finite measure assumption implies that $-\Delta_\Omega^{\rm D}$ has discrete spectrum, each of its eigenvalues has finite multiplicity and their only accumulation point is $+\infty$. Let $\{\psi_n\}_{n\geq 1}$ be a complete orthonormal set of eigenfunctions of $-\Delta_\Omega^{\rm D}$ and $\{\lambda_n\}_{n\geq 1}$ the corresponding eigenvalues.

	We will apply \eqref{itm: Abstract Dir} in Lemma~\ref{lem: abstract identities} with $\mathcal H= L^2(\Omega), \hat{\mathcal H}= L^2(\R^d)$, $L = -\Delta_\Omega^{\rm D}$, $\hat L =-\Delta$, and $J\colon L^2(\Omega) \to L^2(\R^d)$ defined by
	\begin{equation*}
		J\phi(x) = \begin{cases}
			\phi(x) & \mbox{for }x\in \Omega\,,\\
			0 & \mbox{for }x\in \Omega^\complement\,.
		\end{cases}
	\end{equation*}
	Note that $L = J^*\hat L J$ in the sense of quadratic forms.

	In the notation of Lemma~\ref{lem: abstract identities} we have for any $\mu \geq 0$
	\begin{align*}
		E(\mu) = \sum_{n: \lambda_n<\mu} \langle \psi_n, \cdot \rangle \psi_n\quad \mbox{and}\quad 
		\hat E(\mu) = \mathcal{F}^{-1}(\1_{|\,\cdot\,|^2<\mu} \widehat{(\,\cdot\,)})\,.
	\end{align*}
	By Lemma~\ref{lem: abstract identities} \eqref{itm: Abstract Dir} it holds that
	\begin{align*}
		\Tr(-\Delta_\Omega^{\rm D}-\Lambda)_\limminus = \int (\lambda-\Lambda)_\limminus d\Tr J^*\hat E(\lambda)J - \mathcal{R}_<'-\mathcal{R}_>'\,.
	\end{align*}

    \medskip

    \emph{Step 2.} We compute the first term on the right. For any complete orthogonal set $\{\phi_j\}_{j\geq 1}\subset L^2(\Omega)$ the fact that the Fourier transform is unitary combined with the Pythagorean theorem implies
	\begin{align*}
		\Tr J^*\hat E(\lambda)J
		&=
		\sum_{j\geq 1}\langle J\phi_j, \hat E(\lambda)J\phi_j\rangle\\
		&= 
		\sum_{j\geq 1}\int_{|\xi|^2<\lambda} |\hat \phi_j(\xi)|^2\,d\xi\\
		&=
		(2\pi)^{-d}\sum_{j\geq 1}\int_{|\xi|^2<\lambda} |\langle \phi_j, e^{i(\cdot)\cdot\xi}\rangle_{L^2(\Omega)}|^2\,d\xi\\
		&=
		(2\pi)^{-d}\int_{|\xi|^2<\lambda} \|e^{i(\cdot)\cdot\xi}\|^2_{L^2(\Omega)}\,d\xi\\
		&= (2\pi)^{-d}|\Omega|\omega_d \lambda^{d/2}\,.
	\end{align*}
	Therefore,
	\begin{equation*}
		\int (\lambda-\Lambda)_\limminus d\Tr J^*\hat E(\lambda)J 
		 = L_{1,d}^{\rm sc}|\Omega|\Lambda^{1+d/2}\,.
	\end{equation*}

    \medskip

    \emph{Step 3.} Combining the formula from Step 1 with the computation in Step 2, we see that the Berezin--Li--Yau inequality follows from the trivial estimates $\mathcal R_>'\geq0, \mathcal R_<'\geq 0$.

	To do better, we provide a positive lower bound for $\mathcal R'_<$. In the case under consideration the fact that $\{\psi_n\}_{n\geq 1}$ is a complete orthonormal set in $L^2(\Omega)$ entails that
	\begin{align*}
		\Tr E(\Lambda) J^*\hat E(\lambda)J E(\Lambda)
		&=
		\sum_{n}\langle \psi_n, E(\Lambda) J^*\hat E(\lambda)J E(\Lambda)\psi_n\rangle_{L^2(\Omega)}\\
		&=
		\sum_{n:\lambda_n<\Lambda}\langle \psi_n, J^*\hat E(\lambda)J \psi_n\rangle_{L^2(\Omega)}\\
		&=
		\sum_{n:\lambda_n<\Lambda}\int_{|\xi|^2<\lambda}|\hat \psi_n(\xi)|^2\,d\xi\\
		&= \Tr(-\Delta_\Omega^{\rm D}-\Lambda)_\limminus^0- \sum_{n:\lambda_n<\Lambda}\int_{|\xi|^2>\lambda}|\hat \psi_n(\xi)|^2\,d\xi\,.
	\end{align*}
	We recall that $\Tr(-\Delta_\Omega^{\rm D}-\Lambda)_\limminus^0$ denotes the number of eigenvalues $\lambda_n<\Lambda$.
	Therefore, by an integration by parts,
	\begin{align*}
		\mathcal{R}'_< = \sum_{\lambda_n<\Lambda}\int_{\Lambda}^\infty \int_{|\xi|^2>\lambda} |\hat\psi_n(\xi)|^2\,d\xi d\lambda\,.
	\end{align*}
    By Theorem~\ref{thm: uncertainty principle Laplacian} and the normalization $\|\psi_n\|_{L^2(\Omega)} = 1$ we have
    $$
    \int_{|\xi|^2>\lambda} |\hat\psi_n(\xi)|^2\,d\xi \geq c e^{-c'\rho_\theta(\Omega)\sqrt{\lambda}} \,.
    $$
    Thus, bounding the resulting integral as in \cite{FrankLarson_Convex24}, we find
	\begin{align*}
		\mathcal{R}'_< &\geq \sum_{\lambda_n<\Lambda}\int_{\Lambda}^\infty c e^{-c'\rho_\theta(\Omega)\sqrt{\lambda}} d\lambda\\
		&\geq c \Tr(-\Delta_\Omega^{\rm D}-\Lambda)_\limminus^0 \int_{\Lambda}^{4\Lambda} e^{-c'\rho_\theta(\Omega)\sqrt{\lambda}}\,d\lambda\\
		&\geq 
		3c \Tr(-\Delta_\Omega^{\rm D}-\Lambda)_\limminus^0  \Lambda e^{-2c'\rho_\theta(\Omega)\sqrt{\Lambda}}\,.
	\end{align*}

	Since $\mathcal{R}_>'\geq 0$ we conclude that
	\begin{equation*}
		\Tr(-\Delta_\Omega^{\rm D}-\Lambda)_\limminus \leq L_{1,d}^{\rm sc}|\Omega|\Lambda^{1+d/2} - 3c \Tr(-\Delta_\Omega^{\rm D}-\Lambda)_\limminus^0  \Lambda e^{-2c'\rho_\theta(\Omega)\sqrt{\Lambda}}\,.
	\end{equation*}

    To finish the proof we distinguish two cases. The first case is where
    $$
    \Tr(-\Delta_\Omega^{\rm D}-\Lambda)_\limminus^0 \geq \frac1{1+3c} L_{1,d}^{\rm sc}|\Omega|\Lambda^{d/2} \,.
    $$
    In this case, the previous estimate implies that
    $$
    \Tr(-\Delta_\Omega^{\rm D}-\Lambda)_\limminus \leq L_{1,d}^{\rm sc}|\Omega|\Lambda^{1+d/2} \left( 1 - \frac{3c}{1+3c} e^{-2c'\rho_\theta(\Omega)\sqrt{\Lambda}} \right).
    $$
    Meanwhile, when
    $$
    \Tr(-\Delta_\Omega^{\rm D}-\Lambda)_\limminus^0  < \frac1{1+3c} L_{1,d}^{\rm sc}|\Omega|\Lambda^{d/2} \,,
    $$
    we have
    \begin{align*}
        \Tr(-\Delta_\Omega^{\rm D}-\Lambda)_\limminus
        & = \int_0^\Lambda \Tr(-\Delta_\Omega^{\rm D} - \lambda)_\limminus^0\,d\lambda \\
        & \leq \Lambda \Tr(-\Delta_\Omega^{\rm D} - \Lambda)_\limminus^0 \\
        & < \frac1{1+3c} L_{1,d}^{\rm sc}|\Omega|\Lambda^{1+d/2} \\
        & = L_{1,d}^{\rm sc}|\Omega|\Lambda^{1+d/2} \left( 1 - \frac{3c}{1+3c} \right).
    \end{align*}
    Thus, in both cases,
    $$
    \Tr(-\Delta_\Omega^{\rm D}-\Lambda)_\limminus \leq L_{1,d}^{\rm sc}|\Omega|\Lambda^{1+d/2} \left( 1 - \frac{3c}{1+3c} e^{-2c'\rho_\theta(\Omega)\sqrt{\Lambda}} \right).
    $$
    This proves Theorem \ref{thm: main Laplace inradius} for $\gamma=1$ in the Dirichlet case.
\end{proof}

\begin{proof}[Proof of Theorem~\ref{thm: main Laplace inradius} for $-\Delta_\Omega^{\rm N}, \gamma=1$]
    Let $\Omega\subset\R^d$ be an open set of finite measure.
    
    \medskip

    \emph{Step 1.} Let $\Sigma := \inf \sigma_{\rm ess}(-\Delta_\Omega^{\rm N})$. If $\Lambda>\Sigma$ the left-hand side of the claimed inequality is infinite and thus the inequality is trivially true. Therefore, in the rest of the proof we assume that $\Lambda \leq \Sigma$. Let $\{\lambda_n\}_{n\geq 1}$ be the set of eigenvalues of $-\Delta_\Omega^{\rm N}$ that are less than $\Lambda\leq \Sigma$ and let $\{\psi_n\}_{n\geq 1}$ be an orthonormal set of eigenfunctions corresponding to these eigenvalues.

	We aim to apply \eqref{itm: Abstract Neu} of Lemma~\ref{lem: abstract identities} with $\mathcal H= L^2(\Omega), \hat{\mathcal H}= L^2(\R^d)$, $L = -\Delta_\Omega^{\rm N}$, $\hat L =-\Delta$, and $J\colon L^2(\Omega) \to L^2(\R^d)$ defined as in the proof of the Dirichlet case.

	In the notation of Lemma~\ref{lem: abstract identities} we have for any $\mu \leq \Lambda$
	\begin{align*}
		E(\mu) = \sum_{n: \lambda_n<\mu} \langle \psi_n,\cdot \rangle \psi_n
	\end{align*}
	and, as before, for all $\mu \geq 0$
	\begin{align*}
		\hat E(\mu) = \mathcal{F}^{-1}(\1_{|\,\cdot\,|^2<\mu} \widehat{(\,\cdot\,)})\,.
	\end{align*}
    Assuming for the moment that $\hat E(\Lambda)J(-\Delta_\Omega^{\rm N} -\Lambda)J^*\hat E(\Lambda)$ is trace class, we deduce from Lemma~\ref{lem: abstract identities} \eqref{itm: Abstract Neu} that
	\begin{align*}
		\Tr(-\Delta_\Omega^{\rm N}-\Lambda)_\limminus = -\int (\lambda-\Lambda) d\Tr \hat E(\Lambda)JE(\lambda)J^*\hat E(\Lambda) + \mathcal{R}_<+\mathcal{R}_>\,.
	\end{align*}

    \medskip

    \emph{Step 2.} We compute the first term on the right and, showing that is finite, will justify that $\hat E(\Lambda)J(-\Delta_\Omega^{\rm N} -\Lambda)J^*\hat E(\Lambda)$ is trace class. We note that $\ran(J^*\hat E(\Lambda))\subset H^1(\Omega) = \dom(\sqrt{-\Delta_\Omega^{\rm N}+1})$. Hence, for any complete orthogonal set $\{\phi_j\}_{j\geq 1}\subset L^2(\R^d)$ we have
	\begin{align*}
		\int (\lambda-\Lambda) d\Tr \hat E(\Lambda)JE(\lambda)J^*\hat E(\Lambda)
		&=
		\sum_{j\geq 1} \int (\lambda-\Lambda)d\langle J^* \hat E(\Lambda) \phi_j, E(\lambda)J^*\hat E(\Lambda)\phi_j\rangle\\
		&=
		\sum_{j\geq 1} \bigl(\|\nabla (\hat E(\Lambda) \phi_j)\|_{L^2(\Omega)}^2-\Lambda\|\hat E(\Lambda) \phi_j\|_{L^2(\Omega)}^2\bigr)\,.
	\end{align*}
	Since $\{\hat \phi_j\}_{j\geq 1}$ is an orthonormal basis for $L^2(\R^d)$, the Pythagorean theorem implies
	\begin{align*}
		\sum_{j\geq 1}\|\hat E(\Lambda)\phi_j\|_{L^2(\Omega)}^2 
		&= 
		(2\pi)^{-d}\int_\Omega \sum_{j\geq 1}|\langle \1_{|\cdot|^2<\Lambda}e^{-ix \cdot (\cdot)}, \hat\phi_j\rangle_{L^2(\R^d)}|^2\,dx\\
		&=
		(2\pi)^{-d}\int_\Omega \|\1_{|\cdot|^2<\Lambda}e^{-ix \cdot (\cdot)}\|_{L^2(\R^d)}^2\,dx\\
		&= (2\pi)^{-d} |\Omega| \int_{|\xi|^2<\Lambda}\,d\xi
	\end{align*}
	and
	\begin{align*}
		\sum_{j\geq 1}\|\nabla(\hat E(\Lambda)\phi_j)\|_{L^2(\Omega)}^2 
		&= 
		(2\pi)^{-d}\int_\Omega \sum_{j\geq 1}\sum_{k=1}^d|\langle (\cdot)_k\1_{|\cdot|^2<\Lambda}e^{-ix \cdot (\cdot)}, \hat\phi_j\rangle_{L^2(\R^d)}|^2\,dx\\
		&=
		(2\pi)^{-d}\int_\Omega \bigl\||\cdot| \1_{|\cdot|^2<\Lambda}e^{-ix \cdot (\cdot)}\bigr\|_{L^2(\R^d)}^2\,dx\\
		&= (2\pi)^{-d}|\Omega| \int_{|\xi|^2<\Lambda}|\xi|^2\,d\xi\,.
	\end{align*}
    These two formulas show that the operators $\hat E(\Lambda)J(-\Delta_\Omega^{\rm N})J^*\hat E(\Lambda)$ and $\hat E(\Lambda)JJ^*\hat E(\Lambda)$ are trace class under our assumptions. Moreover, they show
	\begin{align*}
		-\int (\lambda-\Lambda) d\Tr \hat E(\Lambda)JE(\lambda)J^*\hat E(\Lambda)
		&=
		-(2\pi)^{-d}|\Omega|\int_{|\xi|^2<\Lambda}(|\xi|^2-\Lambda)\,d\xi\\
		&= L_{1,d}^{\rm sc}|\Omega|\Lambda^{1+d/2}\,.
	\end{align*}

    \medskip

    \emph{Step 3.} We observe that the Kr\"oger inequality follows from the trivial estimates $\mathcal R_<\geq 0, \mathcal R_>\geq 0$. To do better, we provide a positive lower bound for $\mathcal R_<$.

	For $\lambda \leq \Lambda \leq \Sigma$ and any complete ON set $\{\phi_j\}_{j\geq 1}$we find
	\begin{align*}
		\Tr \hat E(\Lambda)^\perp J E(\lambda)J^* \hat E(\Lambda)^\perp
		&=
		\sum_{j\geq 1}\langle \hat E(\Lambda)^\perp \phi_j,  E(\lambda) \hat E(\Lambda)^\perp\phi_j\rangle_{L^2(\Omega)}\\
		&=
		\sum_{j\geq 1}\sum_{n:\lambda_n<\lambda}|\langle \hat E(\Lambda)^\perp \phi_j,  \psi_n\rangle_{L^2(\Omega)}|^2\\
		&=
		\sum_{j\geq 1}\sum_{n:\lambda_n<\lambda}|\langle  \hat\phi_j,  \1_{|\cdot|^2>\Lambda}\hat \psi_n\rangle_{L^2(\R^d)}|^2\\
		&= \sum_{n:\lambda_n<\lambda}\int_{|\xi|^2>\Lambda}|\hat \psi_n(\xi)|^2\,d\xi\,.
	\end{align*}
	Therefore, in the current case we find
	\begin{align*}
		\mathcal R_< 
		&= 
			\int (\lambda-\Lambda)_\limminus \Tr \hat E(\Lambda)^\perp J E(\lambda)J^* \hat E(\Lambda)^\perp\\
		&= \sum_{n:\lambda_n<\Lambda}(\Lambda-\lambda_n)\int_{|\xi|^2>\Lambda}|\hat \psi_n(\xi)|^2\,d\xi\,.
	\end{align*}
	By Theorem~\ref{thm: uncertainty principle Laplacian} and the normalization $\|\psi_n\|_{L^2(\Omega)}=1$ we have
    $$
    \int_{|\xi|^2>\Lambda}|\hat \psi_n(\xi)|^2\,d\xi \geq c e^{-c' \rho_\theta(\Omega) \sqrt\Lambda} \,.
    $$
    Combining this with Kr\"oger's inequality yields
	\begin{align*}
		\mathcal R_< 
		\geq 
		c e^{-c'\rho_\theta(\Omega)\sqrt{\Lambda}}\Tr(-\Delta_\Omega^{\rm N}-\Lambda)_\limminus
		\geq 
		c L_{1,d}^{\rm sc}|\Omega| \Lambda^{1+d/2}e^{-c'\rho_\theta(\Omega)\sqrt{\Lambda}} \,.
	\end{align*}
	Since $\mathcal{R}_>\geq 0$ we have arrived at the inequality
	\begin{equation*}
		\Tr(-\Delta_\Omega^{\rm N}-\Lambda)_\limminus \geq L_{1,d}^{\rm sc}|\Omega|\Lambda^{1+d/2}(1+c e^{-c'\rho_\theta(\Omega)\sqrt{\Lambda}})\,,
	\end{equation*}
	which completes the proof of Theorem \ref{thm: main Laplace inradius} for $\gamma=1$ in the Neumann case.
\end{proof}

\begin{proof}[Proof of Theorem~\ref{thm: main Landau inradius} for $H_\Omega^{\rm D}$, $\gamma=1$]
	Let $\Omega\subset\R^2$ be an open set of finite measure and let $B>0$. 
    
    \medskip

    \emph{Step 1.} Recall that the finite measure assumption implies that $H_\Omega^{\rm D}$ has discrete spectrum, each of its eigenvalues has finite multiplicity and their only accumulation point is $+\infty$. Let $\{\psi_n\}_{n\geq1}$ be the eigenfunctions of $H_\Omega^{\rm D}$ and $\{\lambda_n\}_{n\geq 1}$ the corresponding eigenvalues. 

	As in the proof for the Dirichlet Laplace operator we aim to apply \eqref{itm: Abstract Dir} of Lemma~\ref{lem: abstract identities}. In this case we choose $\mathcal H=L^2(\Omega), \hat{\mathcal H} = L^2(\R^2), L=H_\Omega^{\rm D}, \hat L = H,$ and $J$ defined as before through extension by zero. Note that in the sense of quadratic forms $H_\Omega^{\rm D} = J^*HJ$.

	We have for any $\mu\geq 0$
	\begin{equation*}
		E(\mu) = \sum_{n:\lambda_n<\mu}\langle \psi_n, \cdot\rangle \psi_n \quad \mbox{and}\quad \hat E(\mu) = \sum_{k: B(2k-1)<\mu} \Pi_k
	\end{equation*}

	By Lemma~\ref{lem: abstract identities} \eqref{itm: Abstract Dir},
	$$
        \Tr (H_\Omega^{\rm D}-\Lambda)_\limminus  = \int (\lambda - \Lambda)_\limminus  \,d\Tr J^* \hat E(\lambda) J - \mathcal R_<' - \mathcal R_>'\,.
    $$

    \medskip
    
    \emph{Step 2.}
    For the first term on the right side we have
    \begin{align*}
     	\int (\lambda - \Lambda)_\limminus  \,d\Tr J^* \hat E(\lambda) J 
     	&= \sum_{k:B(2k-1)<\Lambda}(\Lambda-B(2k-1))\Tr(\1_\Omega \Pi_k\1_\Omega) \\
     	&= \sum_{k:B(2k-1)<\Lambda}(\Lambda-B(2k-1))\int_\Omega \Pi_k(x, x)\,dx\\
     	& = \frac{B}{2\pi}\,|\Omega|\sum_{k:B(2k-1)<\Lambda}(\Lambda-B(2k-1))\,,
    \end{align*} 
    where we used the fact that the integral kernel of $\Pi_k$ satisfies $\Pi_k(x,x)=B/(2\pi)$ for all $x\in\R^2$.

    \emph{Step 3.}
    Combining the formula from Step 1 with the computation in Step 2, we obtain
	\begin{equation}\label{eq: extracting main term}
		\Tr(H_\Omega^{\rm D} - \Lambda)_\limminus  = |\Omega|\, \frac{B}{2\pi} \sum_{k=1}^\infty (B(2k-1)-\Lambda)_\limminus - \mathcal R'_< - \mathcal R'_> \,.
	\end{equation}
	The inequality of Erd\H{o}s--Loss--Vougalter follows by bounding $\mathcal R'_< \geq 0$ and $\mathcal R'_>\geq 0$.

To get an improved inequality, we will still drop the term $\mathcal R_>'$, but we shall provide a positive lower bound for $\mathcal R_<'$ . We write
\begin{align*}
	\mathcal R_<' &= \sum_{n: \lambda_n<\Lambda}\int (\lambda-\Lambda)_\limplus \, d \Tr E(\Lambda)  J^* \hat E(\lambda) J E(\Lambda)\\
	&=
	\sum_{n:\lambda_n<\Lambda} \int (\lambda-\Lambda)_\limplus \, d \langle \psi_n, J^* \hat E(\lambda)J\psi_n\rangle\\
	&= \sum_{n:\lambda_n<\Lambda}\sum_{k\geq 1}(B(2k-1)-\Lambda)_\limplus \|\Pi_k \psi_n\|_{L^2(\R^2)}^2\\
	& = \int_0^\infty \sum_{n:\lambda_n<\Lambda} \sum_{k:B(2k-1)\geq \Lambda+\eta} \|\Pi_k\psi_n\|_{L^2(\R^2)}^2 \,d\eta \,.
\end{align*}
For fixed $\eta>0$ we bound
$$
\sum_{n:\lambda_n<\Lambda} \sum_{k:B(2k-1)\geq \Lambda+\eta} \|\Pi_k\psi_n\|_{L^2(\R^2)}^2
\geq \Tr(H_\Omega^{\rm D}-\Lambda)_\limminus^0 \inf_{n\geq 1} \sum_{k:B(2k-1)\geq \Lambda+\eta} \|\Pi_k\psi_n\|_{L^2(\R^2)}^2 \,.
$$
Now by Theorem~\ref{thm: uncertainty principle Landau} there are constants $c, c'>0$ depending only on $\theta$ such that
$$
\sum_{k:B(2k-1)\geq\Lambda+\eta} \|\Pi_k\psi_n\|_{L^2(\R^2)}^2 \geq c e^{-c'(\rho_\theta(\Omega)\sqrt{\Lambda+\eta} + \rho_\theta(\Omega)^2 B)} \,.
$$
We obtain, as in \cite{FrankLarson_Convex24} and the proof for $-\Delta_\Omega^{\rm D}$,
\begin{equation}\label{eq: R lower bound}
\begin{aligned}
	\mathcal R_<' 
	&\geq 
		c \Tr(H_\Omega^{\rm D}-\Lambda)_\limminus^0 \int_0^\infty e^{-c'(\rho_\theta(\Omega)\sqrt{\Lambda+\eta} + \rho_\theta(\Omega)^2 B)}\,d\eta\\
	&\geq 
		c \Tr(H_\Omega^{\rm D}-\Lambda)_\limminus^0 \int_0^{3\Lambda} e^{-c'(2\rho_\theta(\Omega)\sqrt{\Lambda} + \rho_\theta(\Omega)^2 B)}\,d\eta\\
	&\geq
		3c \Lambda \Tr(H_\Omega^{\rm D}-\Lambda)_\limminus^0 e^{-c'(2\rho_\theta(\Omega)\sqrt{\Lambda} + \rho_\theta(\Omega)^2 B)} \,.
\end{aligned}
\end{equation}

Now we distinguish two cases according to whether $\Tr(H_\Omega^{\rm D}-\Lambda)_\limminus^0$ exceeds 
$$
\frac1{1+3c} \Lambda^{-1} |\Omega| \frac B{2\pi} \sum_{k=1}^\infty (B(2k-1)-\Lambda)_\limminus
$$
or not. Arguing similarly as in the proof of Theorem \ref{thm: main Laplace inradius}, we find that in either case
$$
\Tr (H_\Omega^{\rm D}-\Lambda)_\limminus  \leq |\Omega|\frac{B}{2\pi}\sum_{k=1}^\infty (B(2k-1)-\Lambda)_\limminus \left( 1 - \frac{3c}{1+3c} \, e^{-c'(2\sqrt{\lvert \Omega\rvert\Lambda} + \lvert \Omega\rvert B )} \right).
$$
This proves Theorem \ref{thm: main Landau inradius} for $\gamma=1$ in the Dirichlet case.
\end{proof}

\begin{proof}[Proof of Theorem~\ref{thm: main Landau inradius} for $H_\Omega^{\rm N}, \gamma=1$]
    Let $\Omega\subset\R^2$ be an open set of finite measure and let $B>0$.

    \medskip

    \emph{Step 1.} Let $\Sigma = \inf \sigma_{ess.}(H_\Omega^{\rm N})$. If $\Lambda>\Sigma$ the left-hand side of the inequality in the statement is infinite and the inequality is trivially true. Therefore, we in the rest of the proof assume that $\Lambda \leq \Sigma$. Let $\{\lambda_n\}_{n\geq 1}$ be the set of eigenvalues of $H_\Omega^{\rm N}$ that are less than $\Lambda\leq \Sigma$ and let $\{\psi_n\}_{n\geq 1}$ be an orthonormal set of eigenfunctions corresponding to these eigenvalues.

	We aim to apply \eqref{itm: Abstract Neu} of Lemma~\ref{lem: abstract identities} with $\mathcal H= L^2(\Omega), \hat{\mathcal H}= L^2(\R^2)$, $L = H_\Omega^{\rm N}$, $\hat L =H$, and $J$ defined as in the earlier proofs.

	In the notation of Lemma~\ref{lem: abstract identities} we have for any $\mu \leq \Lambda$
	\begin{align*}
		E(\mu) = \sum_{n: \lambda_n<\mu} \langle \psi_n,\cdot \rangle \psi_n
	\end{align*}
	and for all $\mu \geq 0$
	\begin{align*}
		\hat E(\mu) = \sum_{k: B(2k-1)<\mu} \Pi_k\,.
	\end{align*}
    Assuming for the moment that $\hat E(\Lambda)J(H_\Omega^{\rm N} -\Lambda)J^*\hat E(\Lambda)$ is trace class, we deduce from Lemma~\ref{lem: abstract identities} \eqref{itm: Abstract Neu} that
	\begin{align*}
		\Tr(H_\Omega^{\rm N}-\Lambda)_\limminus = -\int (\lambda-\Lambda) d\Tr \hat E(\Lambda)JE(\lambda)J^*\hat E(\Lambda) + \mathcal{R}_<+\mathcal{R}_>\,.
	\end{align*}

    \medskip

    \emph{Step 2.}
	We compute the first term on the right, in particular, showing that it is finite and thereby also justifying that $\hat E(\Lambda)J(H_\Omega^{\rm N} -\Lambda)J^*\hat E(\Lambda)$ is trace class. We note that
    \begin{align*}
		\hat E(\Lambda)JE(\lambda)J^*\hat E(\Lambda) 
		&=
		\sum_{k,k':B(2k-1)<\Lambda, B(2k'-1)<\Lambda} \Pi_k JE(\lambda)J^*\Pi_{k'} \,,
	\end{align*}
	Let $\{\phi_{j}\}_{j\geq 1}\subset L^2(\R^d)$ be an orthonormal basis of $L^2(\R^2)$ satisfying $\Pi_k \phi_j \in \{0, \phi_j\}$ for each $k,j\geq 1$. Since $\ran J^*\Pi_k\subset H^1_A(\Omega) = \dom(\sqrt{H_\Omega^{\rm N}+1})$ we have
	\begin{align*}
		\int (\lambda-\Lambda) d\Tr \hat E(\Lambda)JE(\lambda)J^*\hat E(\Lambda)
		&=
		\sum_{k:B(2k-1)<\Lambda} \sum_{j\geq 1}\int (\lambda-\Lambda)d\langle J^* \Pi_k\phi_j, E(\lambda)J^*\Pi_k\phi_j\rangle\\
		&=
		\sum_{k:B(2k-1)<\Lambda} \sum_{j\geq 1} \bigl(\|\nabla_A \Pi_k\phi_{j}\|_{L^2(\Omega)}^2-\Lambda\|\Pi_k\phi_{j}\|_{L^2(\Omega)}^2\bigr)\,.
	\end{align*}
	Since $\{\phi_j\}_{j\geq 1}$ is an orthonormal basis for $L^2(\R^2)$ and $\Pi_k(x, x)= B/(2\pi)$ we have
	\begin{align*}
		\sum_{j\geq 1}\|\Pi_k\phi_j\|_{L^2(\Omega)}^2 
		= 
		\Tr(\1_\Omega \Pi_k \1_\Omega)
		= \frac{B}{2\pi}|\Omega|\,.
	\end{align*}
	The argument for the term involving $\nabla_A$ is a bit more complicated. Denote by $\Psi_{k,x}$ the restriction of $\Pi_k(\cdot,x)$ to $\Omega$. We have for any $\psi\in L^2(\R^2)$ and $x \in \Omega$
	$$
		\Pi_k \psi(x) = \int_{\R^2} \Psi_{k,x'}(x)\psi(x')\,dx' = \langle \Psi_{k,\cdot}(x), \psi\rangle_{L^2(\R^2)}
	$$
	and by smoothness of the kernel
	$$
		\nabla_A \Pi_k\psi(x) =\int_{\R^2} \nabla_{A,x}\Psi_{k,x'}(x)\psi(x')\,dx' = \langle \nabla_{A,x}\Psi_{k, \cdot}(x),\psi\rangle_{L^2(\R^2)}
	$$
	Here $\nabla_{A,x}$ denotes the magnetic gradient acting with respect to the $x$ variable.
	The completeness of $\{\phi_j\}_{j\geq 1}$ implies that
	\begin{equation*}
		\sum_{j\geq 1} \|\nabla_A \Pi_k\phi_{j}\|_{L^2(\Omega)}^2 = \int_\Omega\int_{\R^2}|\nabla_{A,x}\Psi_{k,x'}(x)|^2\,dx'dx \,.
	\end{equation*}
	Now, by an identity proved in \cite{Frank_Semigroup} we have for every $x\in\Omega$,
	$$
		\int_{\R^2} |(\nabla_{A,x} \Psi_{k,x'})(x)|^2 \,dx' = \frac{B}{2\pi} \, B(2k-1) \,.
	$$

	Consequently, we have shown that
	\begin{align*}
		-\int (\lambda-\Lambda) d\Tr \hat E(\Lambda)JE(\lambda)J^*\hat E(\Lambda)
		&=
		|\Omega|\frac{B}{2\pi}\sum_{k=1}^\infty(B(2k-1)-\Lambda)_\limminus \,.
	\end{align*}
    We note that the argument that we have just given also shows that the operators $\hat E(\Lambda)JH_\Omega^{\rm N} J^*\hat E(\Lambda)$ and $\hat E(\Lambda)JJ^*\hat E(\Lambda)$ are trace class.

    \medskip

    \emph{Step 3.}     
	The semiclassical inequality obtained in \cite{Frank_Semigroup} thus follows from the formula in Step 1 and the computation in Step 2 via the trivial estimates $\mathcal R_<\geq 0, \mathcal R_>\geq 0$. To do better we yet again provide a positive lower bound for $\mathcal R_<$. 

	For $\lambda \leq \Lambda \leq \Sigma$ and any complete ON set $\{\phi_j\}_{j\geq 1}$ we find
	\begin{align*}
		\Tr \hat E(\Lambda)^\perp J E(\lambda)J^* \hat E(\Lambda)^\perp
		&=
		\sum_{j\geq 1}\langle \hat E(\Lambda)^\perp \phi_j,  E(\lambda) \hat E(\Lambda)^\perp\phi_j\rangle_{L^2(\Omega)}\\
		&=
		\sum_{j\geq 1}\sum_{n:\lambda_n<\lambda}|\langle \hat E(\Lambda)^\perp \phi_j,  \psi_n\rangle_{L^2(\Omega)}|^2\\
		&=
		\sum_{j\geq 1}\sum_{n:\lambda_n<\lambda}\sum_{k: B(2k-1)\geq \Lambda}|\langle  \Pi_k\phi_j,  \psi_n\rangle_{L^2(\Omega)}|^2\\
		&=
		\sum_{j\geq 1}\sum_{n:\lambda_n<\lambda}\sum_{k: B(2k-1)\geq \Lambda}|\langle  \phi_j, \Pi_k \psi_n\rangle_{L^2(\R^2)}|^2\\
		&= \sum_{n:\lambda_n<\lambda}\sum_{k: B(2k-1)\geq \Lambda}\|\Pi_k \psi_n\|_{L^2(\R^2)}^2 \,.
	\end{align*}
	Therefore, we have
	\begin{align*}
		\mathcal R_< 
		&= 
			\int (\lambda-\Lambda)_\limminus \Tr \hat E(\Lambda)^\perp J E(\lambda)J^* \hat E(\Lambda)^\perp\\
		&= \sum_{n:\lambda_n<\Lambda}(\Lambda-\lambda_n)\biggl(\sum_{k: B(2k-1)\geq \Lambda}\|\Pi_k \psi_n\|_{L^2(\R^2)}^2\biggr)\,.
	\end{align*}
	By Theorem~\ref{thm: uncertainty principle Landau} and the semiclassical inequality for $H_\Omega^{\rm N}$ proved in \cite{Frank_Semigroup}
	\begin{align*}
		\mathcal R_< 
		&\geq 
		c e^{-c'(\rho_\theta(\Omega)\sqrt{\Lambda}+\rho_\theta(\Omega)^2B)}\sum_{n:\lambda_n<\Lambda}(\Lambda-\lambda_n)\\
		&=
		c e^{-c'(\rho_\theta(\Omega)\sqrt{\Lambda}+\rho_\theta(\Omega)^2B)}\Tr(H_\Omega^{\rm N}-\Lambda)_\limminus\\
		&\geq 
		c |\Omega| \frac{B}{2\pi} e^{-c'(\rho_\theta(\Omega)\sqrt{\Lambda}+\rho_\theta(\Omega)^2B)}\sum_{k=1}^\infty(B(2k-1)-\Lambda)_\limminus \,.
	\end{align*}
	Since $\mathcal{R}_>\geq 0$ we have arrived at the claimed inequality
	\begin{equation*}
		\Tr(H_\Omega^{\rm N}-\Lambda)_\limminus \geq c |\Omega| \frac{B}{2\pi} \sum_{k=1}^\infty(B(2k-1)-\Lambda)_\limminus\Bigl(1+ce^{-c'(\rho_\theta(\Omega)\sqrt{\Lambda}+\rho_\theta(\Omega)^2B)}\Bigr)\,.
	\end{equation*}
    This proves Theorem \ref{thm: main Landau inradius} for $\gamma=1$ in the Neumann case.
\end{proof}

%%%%%%%%%%%%%%%%%%%%%%%%%%%%%%%%%%%%%%%%

\section{Consequences for Riesz means of higher order}\label{sec:gamma>}

So far, we have proved Theorems \ref{thm: main Laplace inradius} and \ref{thm: main Landau inradius} for $\gamma=1$. In this section we explain how to deduce the result for $\gamma>1$ from this special case. We consider the non-magnetic and magnetic cases simultaneously.

Recall that by the Aizenman--Lieb identity we have for any $\gamma>1$, any $\Lambda\geq 0$ and any selfadjoint operator $L$
\begin{equation}\label{eq: AizenmanLieb}
	\Tr(L-\Lambda)_\limminus^\gamma = \gamma(\gamma-1)\int_0^\Lambda(\Lambda- \lambda)^{\gamma-2}\Tr(L-\lambda)_\limminus \,d\lambda\,.
\end{equation}
Define 
$$
	G^0_\gamma (\Lambda) := L_{\gamma,d}^{\rm sc}\Lambda^{\gamma+d/2} \quad \mbox{and}\quad G^B_\gamma (\Lambda) := \frac{B}{2\pi} \sum_{k=1}^\infty(B(2k-1)-\Lambda)_\limminus^\gamma \quad \mbox{for }B>0\,.
$$
(It is worth noting that $\lim_{B\to 0} G^B_\gamma(\Lambda) = G^0_\gamma(\Lambda)$ when $d=2$.)
For $\gamma>1$ and $B\geq 0$ we have, similarly to \eqref{eq: AizenmanLieb},
$$
G_\gamma^B (\Lambda) = \gamma(\gamma-1) \int_0^\Lambda(\Lambda- \lambda)^{\gamma-2} G^B_1(\lambda)\,d\lambda \,.
$$

Combining these formulas with Theorems \ref{thm: main Laplace inradius} and~\ref{thm: main Landau inradius} we infer that there are constants $c, c'$ depending only on $\theta$ and $d$ such that
\begin{align}
    \label{eq: integrated Riesz bound lap D}
	\Tr(-\Delta^{\rm D}_\Omega-\Lambda)_\limminus^\gamma & \leq |\Omega| G_\gamma^0(\Lambda) - \mathcal R^0 \,, \\
	\label{eq: integrated Riesz bound lap N}
	\Tr(-\Delta^{\rm N}_\Omega-\Lambda)_\limminus^\gamma & \geq |\Omega| G_\gamma^0(\Lambda) + \mathcal R^0 \,,\\
	\label{eq: integrated Riesz bound D}
	\Tr(H^{\rm D}_\Omega-\Lambda)_\limminus^\gamma & \leq |\Omega| G_\gamma^B(\Lambda) - \mathcal R^B \,, \\
	\label{eq: integrated Riesz bound N}
	\Tr(H^{\rm N}_\Omega-\Lambda)_\limminus^\gamma & \geq |\Omega| G_\gamma^B(\Lambda) + \mathcal R^B \,,
\end{align}
where, for $B\geq 0$,
$$
\mathcal R^B := c|\Omega| \gamma(\gamma-1) \int_0^\Lambda(\Lambda- \lambda)^{\gamma-2}G^B_1(\lambda) e^{-c'(\rho_\theta(\Omega)\sqrt{\lambda} + \rho_\theta(\Omega)^2B )} \,d\lambda \,.
$$
Bounding $e^{-c'(\rho_\theta(\Omega)\sqrt{\lambda} + \rho_\theta(\Omega)^2B )} \geq e^{-c'(\rho_\theta(\Omega)\sqrt{\Lambda} + \rho_\theta(\Omega)^2B )}$ in the integrand, we find
$$
\mathcal R^B \geq c\, e^{-c'(\rho_\theta(\Omega)\sqrt{\Lambda} + \rho_\theta(\Omega)^2B )} |\Omega| G^B_\gamma(\Lambda) \,.
$$

Upon inserting this bound into \eqref{eq: integrated Riesz bound lap D}--\eqref{eq: integrated Riesz bound N} we obtain the bounds in Theorems \ref{thm: main Laplace inradius} and \ref{thm: main Landau inradius} for $\gamma>1$. This concludes the proof of these two theorems.

%%%%%%%%%%%%%%%%%%%%%%%%%%%%%%%%%%%%
%%%%%%%%%%%%%%%%%%%%%%%%%%%%%%%%%%%%

\bibliographystyle{amsalpha}

\end{document}